\documentclass[10pt,reqno]{amsart}

\usepackage{graphicx}
\usepackage[dvipsnames]{xcolor}
\usepackage{indentfirst,csquotes,paralist,fancyhdr,etoolbox}
\usepackage{amsmath,amssymb,amsthm}
\usepackage{bm}
\usepackage[authoryear,round]{natbib}
\usepackage{float}
\usepackage[letterpaper,left=0.85in,right=0.85in,top=1in,bottom=1in,includeheadfoot]{geometry}
\usepackage{hyperref}
\usepackage[utf8]{inputenc}
\usepackage[T1]{fontenc}
\usepackage{tikz}
\usetikzlibrary{calc,shapes.geometric,fit,positioning,shapes.misc,arrows.meta,backgrounds}
\usepackage{algorithm}
\usepackage{algpseudocode}
\usepackage{booktabs}
\usepackage{multirow}
\usepackage{placeins}
\usepackage{subcaption}
\usepackage[bottom]{footmisc}

\hypersetup{colorlinks=true,linkcolor=black,citecolor=black,urlcolor=black}
\algtext*{EndFor}
\algtext*{EndIf}

\newtheorem{theorem}{Theorem}
\newtheorem{lemma}{Lemma}

\title{Hierarchical Pod Routing for Modular Transit Systems}
\author{Tina Radvand, Alireza Talebpour, and Yanfeng Ouyang}
\thanks{The Grainger College of Engineering, Department of Civil \& Environmental Engineering, The University of Illinois at Urbana-Champaign. Email: \texttt{radvand2@illinois.edu}, \texttt{ataleb@illinois.edu}, \texttt{yfouyang@illinois.edu}.}

\begin{document}
\begin{abstract}
Modular bus systems allow fixed-route transit services to better match vehicle capacity to passenger loads. Operating modular bus systems requires deciding how pods are assigned to scheduled services, repositioned for subsequent assignments, and parked when not in service. A time-space minimum-cost flow formulation provides an exact benchmark for this problem, but it does not scale well to large urban transit systems with many pod-level routing decisions. This paper proposes a Hierarchical Pod Routing (HPR) framework that solves the problem in two stages. The first stage uses a minimum path-cover formulation to assign service trips to pods while minimizing fleet size. The second stage determines how pods are repositioned and parked between consecutive assignments to minimize empty movement and parking costs. The proposed method is evaluated on benchmark instances constructed from New York City MTA GTFS data, which capture realistic bus schedules and passenger-demand patterns. The results show that HPR remains within $8\%$ of the exact minimum-cost flow objective on the largest instances, while achieving more than a four order of magnitude reduction in runtime and an approximately one order of magnitude reduction in memory use. The results also show that HPR distributes service trips more evenly across pods and reduces the frequency of long idle periods between assignments, while the exact minimum-cost flow formulation produces shorter distances between consecutive assignments. Sensitivity analyses further show that HPR maintains consistent performance across a range of operating-cost assumptions. These findings suggest that HPR can provide a scalable and near-optimal planning approach for modular transit systems.
\end{abstract}
\keywords{Public transportation, Modular transit, Pod routing, Fleet rebalancing, Hierarchical optimization}
\maketitle

\section{Introduction}
\label{sec:Intro}
Public transit systems, especially bus services, are vital for sustainable urban development. They significantly help reduce road congestion~\citep{stopher2004reducing}, reduce emissions~\citep{bel2018evaluation}, and provide fair access to transportation~\citep{venter2018equity}. However, putting these benefits into practice often faces many operational challenges.

In areas with low population density and low demand, traditional fixed-route bus services are often inefficient. They require substantial fuel and labor resources, which makes them costly to operate. To control costs, transit agencies often reduce service frequencies, but these cost-saving measures can make transit less attractive and cause ridership to drop even further, creating a downward cycle. In contrast, high-demand urban areas struggle with peak-hour crowding~\citep{tyrinopoulos2013factors} and longer travel times due to frequent stops~\citep{ha2020unraveling}. In addition, significant costs arise from empty bus repositioning, which is often needed to address directional demand imbalances~\citep{cortes2011integrating}.

Modular bus technology offers a promising alternative to conventional bus operations by addressing some of these limitations. Its design supports non-stop journeys~\citep{khan2025no}, reducing overall travel times and minimizing service disruptions, such as bus bunching~\citep{khan2023application, khan2023bus}. In addition, modular buses can adjust available capacity according to demand, with more pods deployed during peak hours and fewer during off-peak periods. This flexibility can improve fleet utilization while maintaining the desired level of service~\citep{tian2022planning, cheng2024autonomous}.

While modular transit offers considerable potential to address many current transit challenges, its effective implementation requires several operational planning decisions, including service frequency, the number of pods assigned to each service trip, and the cost-efficient routing of pods across the network to meet fluctuating passenger demand.

Existing research on modular bus systems has mainly focused on service design decisions, such as headways and the number of pods assigned to each service trip~\citep{dai2020joint}. Beyond these decisions, operators must also determine how individual pods are assigned to service trips and where they wait between consecutive assignments. This scheduling problem is related to conventional vehicle scheduling and routing. However, modular bus systems can substantially increase the number of routing decisions because each scheduled bus may consist of multiple pods, each of which can be reused independently throughout the service day. Consequently, directly adapting exact time-space network formulations can become computationally expensive for large urban transit systems in terms of both solution time and memory requirements. This creates a need for scalable optimization approaches that reduce the computational burden of exact network-based formulations while maintaining high-quality routing decisions.
 
This study addresses this gap by developing Hierarchical Pod Routing (HPR), a scalable framework for constructing pod operating plans over a fixed service schedule. Given stop-level pod demand along each bus run, HPR first decomposes bus runs into pod routes and then assigns these routes to pods by solving a minimum disjoint path cover problem in a directed acyclic graph (DAG). After the pod assignments are determined, HPR selects waiting locations between consecutive assignments by choosing locations with the lowest combined empty-movement and parking cost. This representation is computationally attractive because HPR avoids solving the full minimum-cost flow problem on a large time-space network; instead, it first determines pod assignments and then selects waiting locations. The first stage can be solved efficiently using GPU-based push-relabel algorithms. We evaluate HPR against the exact minimum-cost flow formulation using benchmark instances constructed from Manhattan bus data, comparing optimality gap, runtime, memory use, number of trips assigned per pod, idle time between consecutive assignments, distance between consecutive assignments, and sensitivity to operating costs. The results show that HPR substantially reduces runtime and memory use while maintaining near-optimal solution quality.

The remainder of this paper is organized as follows. Section~\ref{sec:LitReview} reviews related work on modular transit and vehicle scheduling. Section~\ref{sec:Model} describes the pod routing problem and presents the exact time-space minimum-cost flow formulation. Section~\ref{sec:SolutionApproach} describes the proposed HPR framework. Section~\ref{sec:CompExp} describes the benchmark instances and computational setup. Section~\ref{sec:Results} presents the computational results and sensitivity analyses. Section~\ref{sec:Conclusion} concludes the paper.

\section{Literature Review}
\label{sec:LitReview}
This section reviews literature related to the operational planning of modular transit systems and vehicle assignment in scheduled transit. The review first discusses studies on modular transit service design and operations. It then reviews conventional transit scheduling models that assign vehicles to fixed schedules while accounting for operating costs such as deadheading and waiting. Finally, it discusses methods for linking scheduled trips into vehicle sequences.

Research in operational planning and management of modular transit has largely concentrated on optimizing service-level decisions. These studies explore how modularity enhances service quality by adapting vehicle capacity and dispatching strategies to meet fluctuating demand. For instance, \cite{dakic2021design} used a three-dimensional macroscopic fundamental diagram (3D-MFD) framework to jointly determine service frequency and the composition of modular bus units. Similarly, \cite{dai2020joint} proposed a dynamic programming approach to jointly optimize dispatch headways and vehicle capacities in mixed traffic, minimizing total system cost.

Several studies also explored specific operational designs. \cite{chen2019operational} focused on the operational design of shuttle systems with modular vehicles under oversaturated traffic, optimizing dispatch frequency, composition, and platooning decisions. \cite{chen2021designing} proposed a discrete modeling method for designing transit corridor systems where modular autonomous vehicles adjust their formation dynamically at stations through docking and undocking operations. \cite{pei2021vehicle} presented a mixed-integer nonlinear programming model for optimal vehicle dispatching and modular composition in a network, balancing vehicle operation costs and passenger trip time costs. Additionally, \cite{cao2025optimizing} proposed an integer nonlinear programming (INLP) model that jointly optimizes departure intervals, modular vehicle capacity, short-turning schemes, and decoupling/coupling operations. The model explicitly accounts for modular units that are decoupled or coupled, deadheaded from, or returned to the depot.

While these studies have advanced the modeling of service-level decisions in modular transit, their focus has largely been on how modular service is planned and operated. They generally do not explicitly address the repositioning and scheduling of empty pods between service assignments. These decisions are critical for cost-effective operations because they determine how pods are reused, how far they travel empty, and where they wait between assignments. This leaves between-service routing in modular transit systems an underexplored gap in the current literature.

The scheduling of empty modular pods is related to conventional vehicle scheduling in public transit. In the vehicle scheduling problem, vehicles are assigned to fixed timetabled trips while minimizing fleet and operating costs, including deadheading and waiting. Network-flow models formulated on time-space networks are widely used for vehicle scheduling decisions. \cite{kliewer2006time} formulated the multi-depot multi-vehicle-type bus scheduling problem on a time-space network, assigning buses to timetabled trips while minimizing fleet and operating costs. Their formulation reduces the number of explicit trip-to-trip connections by aggregating feasible connections within the network. More recently, \cite{olsen2022study} investigated flow decomposition methods for electric bus scheduling using aggregated time-space network models to construct and refine vehicle rotations.

For modular transit systems, time-space network formulations can grow substantially in size. In conventional bus scheduling, each bus run is typically served in full by a single bus. In modular transit, by contrast, a single bus run may require multiple pods, and individual pods may join or leave the run at different stops as passenger demand changes along the route. These pod-level transitions increase the number of nodes and edges in the time-space network, limiting the scalability of exact formulations for large urban systems.

A special case of vehicle scheduling focuses on determining the minimum number of vehicles required to cover a fixed schedule. One established approach is the deficit-function method of \cite{ceder1981deficit}, which involves a two-stage process: first, a graphical deficit function is used to determine an initial minimum number of vehicles for a fixed schedule. The second stage then systematically inserts deadheading trips to further reduce this fleet size. Building on this concept, \cite{liu2020using} adapted the deficit function theory to specifically address the minimum fleet size problem for autonomous modular public transit systems. 

Another approach for finding the minimum fleet size is solving maximum matching on a bipartite graph \citep{bertossi1987some}.For example, \cite{wang2007heuristic} proposed a two-phase heuristic for the vehicle scheduling problem with route and fueling time constraints, in which vehicle rotations are first generated and then connected through a bipartite graph model to reduce the number of vehicles required.

Matching-based fleet-sizing approaches have also been used in taxi and on-demand mobility systems. \cite{zhan2016graph} and \cite{vazifeh2018addressing} formulated the minimum fleet problem as a maximum bipartite matching problem and applied it to large-scale taxi trip data. \cite{ye2023min} later established a min-max theorem showing that the minimum fleet size equals the maximum number of pairwise incompatible trips and can be computed through maximum bipartite matching.

Prior work has studied modular transit service-level decisions and conventional vehicle scheduling. However, existing studies do not provide a scalable framework for assigning individual modular pods to bus runs while determining their deadheading and waiting locations between assignments. Exact time-space network formulations can capture these decisions and their associated costs, but their size can grow substantially in modular transit systems. This study addresses this gap by proposing a hierarchical pod routing framework that separates pod service sequencing from between-assignment deadheading and waiting decisions. The framework is evaluated against an exact time-space network formulation to assess its scalability and solution quality.

\section{Model Formulation}\label{sec:Model}

We consider a fixed-route transit system in which modular pods are assigned to scheduled bus runs to satisfy passenger demand. Let \(B\) denote the set of bus runs. Each run \(b \in B\) is characterized by an ordered sequence of stops and scheduled arrival times at those stops. Let \(S_b\) denote the ordered sequence of stops for run \(b\), and let \(d_{b,s} \in \mathbb{Z}_+\) denote the number of pods required at stop \(s \in S_b\) on run $b$. Passengers may move between pods while the bus travels between consecutive stops. Thus, passengers can move from a pod that is no longer required downstream to pods that continue in service before reaching the decoupling stop, allowing the pod to decouple without additional dwell time \citep{khan2025no}. Figure~\ref{fig:bus_run} illustrates an example with four bus runs.
\tikzset{
    styleblue/.style={draw=black, fill=none, thick, double, circle, text=black, minimum size=3mm, font=\scriptsize,inner sep=2pt},
    stylegreen/.style={draw=black, fill=none, very thick, circle, text=black, minimum size=3mm, font=\scriptsize},
    stylered/.style={draw=black, fill=none, very thick, dashed, circle, text=black, minimum size=3mm, font=\scriptsize,inner sep=2pt},
    styleblack/.style={draw=black, fill=none, very thick, solid, circle, text=black, minimum size=3mm, font=\scriptsize,inner sep=2pt},
    stylepurple/.style={draw=black, fill=black, circle, text=black, minimum size=2mm, font=\scriptsize, inner sep=0pt},
    stylefade/.style={draw=lightgray, thick, fill=none, dash dot, circle, text=lightgray, minimum size=3mm, font=\scriptsize }
}

\begin{figure}[tb]
    \centering
    \begin{tikzpicture}
\draw[->] (-0.5,0) -- (0.8*\linewidth,0) node[right]{ };

\foreach \x/\time in {0/13:00, 0.06666666666/13:02, 0.1/13:03, 0.16666666666/13:05, 
    0.23333333333/13:07, 0.33333333333/13:10, 0.43333333333/13:13, 
    0.46666666666/13:14, 0.56666666666/13:17,0.6/13:18, 0.66666666666/13:20, 
    0.76666666666/13:23, 0.8/13:24, 0.83333333333/13:25} 
{
    \draw (\x*\linewidth*0.9, -0.1) -- (\x*\linewidth*0.9, 0.1);
    \draw (\x*\linewidth*0.9 + 0.013*\linewidth, 0.5) node[above, rotate=90, font=\scriptsize] {\time};
}
\node[styleblack] (1) at (0*\linewidth*0.9, -1) {$1$};
\node[styleblack] (4) at (0.06666666666*\linewidth*0.9, -1) {$2$};
\node[styleblack] (7) at (0.16666666666*\linewidth*0.9, -1) {$3$};
\node[styleblack] (9) at (0.23333333333*\linewidth*0.9, -1) {$4$};

\node[styleblack] (11) at (0.6*\linewidth*0.9, -1) {$1$};
\node[styleblack] (13) at (0.66666666666*\linewidth*0.9, -1) {$2$};
\node[styleblack] (15) at (0.76666666666*\linewidth*0.9, -1) {$3$};
\node[styleblack] (16) at (0.83333333333*\linewidth*0.9, -1) {$4$};

\node[styleblack] (17) at (0.43333333333*\linewidth*0.9, -2.5) {$5$};
\node[styleblack] (18) at (0.56666666666*\linewidth*0.9, -2.5) {$6$};
\node[styleblack] (20) at (0.66666666666*\linewidth*0.9, -2.5) {$3$};
\node[styleblack] (22) at (0.8*\linewidth*0.9, -2.5) {$7$};

\node[styleblack] (25) at (0.1*\linewidth*0.9, -4) {$5$};
\node[styleblack] (27) at (0.23333333333*\linewidth*0.9, -4) {$6$};
\node[styleblack] (29) at (0.33333333333*\linewidth*0.9, -4) {$3$};
\node[styleblack] (31) at (0.46666666666*\linewidth*0.9, -4) {$7$};

\node[font=\scriptsize] at (0*\linewidth*0.9, -0.5) {3};
\node[font=\scriptsize] at (0.06666666666*\linewidth*0.9, -0.5) {3};
\node[font=\scriptsize] at (0.16666666666*\linewidth*0.9, -0.5) {2};
\node[font=\scriptsize] at (0.23333333333*\linewidth*0.9, -0.5) {2};

\node[font=\scriptsize] at (0.6*\linewidth*0.9, -0.5) {2};
\node[font=\scriptsize] at (0.66666666666*\linewidth*0.9, -0.5) {2};
\node[font=\scriptsize] at (0.76666666666*\linewidth*0.9, -0.5) {1};
\node[font=\scriptsize] at (0.83333333333*\linewidth*0.9, -0.5) {1};

\node[font=\scriptsize] at (0.43333333333*\linewidth*0.9, -2) {1};
\node[font=\scriptsize] at (0.56666666666*\linewidth*0.9, -2) {2};
\node[font=\scriptsize] at (0.66666666666*\linewidth*0.9, -2) {2};
\node[font=\scriptsize] at (0.8*\linewidth*0.9, -2) {1};

\node[font=\scriptsize] at (0.1*\linewidth*0.9, -3.5) {2};
\node[font=\scriptsize] at (0.23333333333*\linewidth*0.9, -3.5) {3};
\node[font=\scriptsize] at (0.33333333333*\linewidth*0.9, -3.5) {3};
\node[font=\scriptsize] at (0.46666666666*\linewidth*0.9, -3.5) {1};

\path [->, thick] (1) edge (4);
\path [->, thick] (4) edge (7);
\path [->, thick] (7) edge (9);

\path [->, thick, out=0, in=180] (11) edge (13);
\path [->, thick,out=0, in=180] (13) edge (15);
\path [->, thick, out=0, in=180] (15) edge (16);

\path [->, thick, out=0, in=180] (25) edge (27);
\path [->, thick, out=0, in=180] (27) edge (29);
\path [->, thick, out=0, in=180] (29) edge (31);

\path [->, thick, out=0, in=180] (17) edge (18);
\path [->, thick, out=0, in=180] (18) edge (20);
\path [->, thick, out=0, in=180] (20) edge (22);

\end{tikzpicture}
    \caption{Illustration of four bus runs in a transit network. The value above each stop \(s\) represents the required number of pods, \(d_{b,s}\), for the corresponding bus run \(b\).}
    \label{fig:bus_run}
\end{figure}

The objective is to determine pod routes that satisfy pod demand while minimizing the total daily system cost. The total cost consists of three components: fixed pod costs, deadheading costs, and parking costs. The first component is a fixed cost, \(c^{\text{f}}\), incurred for each pod deployed and representing amortized capital costs and recurring expenses such as maintenance, cleaning, and repairs. The second component is a deadheading cost, \(c^{\text{m}}\), incurred per unit of travel time when a pod moves empty between service assignments. The third component is a parking cost, \(c^{\text{p}}_{s}\),incurred per unit of time that a pod waits at stop \(s\). Pods may wait at any bus stop between consecutive service assignments. Dedicated depots can be incorporated as additional waiting locations; however, this study considers bus stops as the only waiting locations.

We formulate the problem as a time-space network that represents the locations and availability of pods over time. The network contains requester, releaser, and standby nodes, together with one source node and one sink node. Each requester, releaser, or standby node is represented by a triplet \((k,s,t)\), where \(k \in \{q,r,w\}\) identifies the node type, s denotes the stop, and \(t\) denotes the time. The value $q$, $r$ and $w$  correspond to requester, releaser, and standby nodes, respectively.

Requester nodes represent locations and times at which pods join a bus run, whereas releaser nodes represent locations and times at which pods leave the run and become available for subsequent assignments. For each bus run, a requester node is created at its first stop with demand equal to the number of pods required at that stop. An additional requester node is created at each intermediate stop where pod demand exceeds that at the preceding stop, with demand equal to the corresponding increase. Similarly, a releaser node is created at the last stop of each run with supply equal to the number of pods required at that stop. An additional releaser node is created at each intermediate stop where pod demand exceeds that at the next stop, with supply equal to the corresponding decrease.

For each requester node $(q,s,t)$, a standby node is created at every stop $s' \in S$ at time $t-t_{s's}$, where $S$ denotes the set of stops in the transit network and $t_{s's}$ is the travel time from $s'$ to $s$. Similarly, for each releaser node $(r,s,t)$, a standby node is created at every stop $s' \in S$ at time $t+t_{ss'}$. These standby nodes represent the times at which a pod may be present at a waiting location immediately before traveling to a requester node or immediately after traveling from a releaser node. Standby nodes whose times fall outside the service horizon are excluded from the network. Finally, one source node and one sink node are added to represent the deployment of pods at the beginning of the service horizon and their removal from the network at its end, respectively.

Edges in the network represent feasible movements and waiting decisions over time and space. A source edge connects the source node to the earliest standby node at each stop. Assignment edges connect standby nodes to requester nodes, representing pods traveling to begin service, whereas release edges connect releaser nodes to standby nodes, representing pods traveling to a waiting location after leaving service. Parking edges connect consecutive standby nodes at the same stop in chronological order and represent pods waiting at that stop during the corresponding time interval. Finally, a sink edge connects the latest standby node at each stop to the sink node.

Each edge is assigned a per-unit-flow cost based on its type. Each source edge has a per-unit-flow cost \(c^{\text{f}}\), representing the daily cost of deploying one pod. An assignment or release edge with travel time \(\Delta t\) has cost \(c^{\text{m}}\Delta t\), whereas a parking edge of duration \(\Delta t\) at stop \(s\) has cost \(c^{\text{p}}_{s}\Delta t\). Sink edges have zero cost. Figure~\ref{fig:timespace} illustrates the time-space network constructed for the four bus runs in Figure~\ref{fig:bus_run}. Algorithm~\ref{alg:timespace} summarizes the construction of the time-space network.
\begin{figure}[tb]
    \centering
    \begin{tikzpicture}
\draw[->] (-0.5,0) -- (0.8*\linewidth,0) node[right]{ };

\foreach \m in {0,...,25} {

    \pgfmathsetmacro{\x}{\m/30}

    \draw (\x*\linewidth*0.9, -0.1) -- (\x*\linewidth*0.9, 0.1);

    \pgfmathtruncatemacro{\minute}{\m}

    \draw (\x*\linewidth*0.9 + 0.013*\linewidth, 0.5)
        node[above, rotate=90, font=\scriptsize]
        {13:\ifnum\minute<10 0\minute\else\minute\fi};
}
\foreach \m in {0,...,25} {

    \pgfmathsetmacro{\x}{\m/30}

    \draw[gray!10, thin]
        (\x*\linewidth*0.9, 0)
        --
        (\x*\linewidth*0.9, -7);
}
\node[styleblue] (1) at (0*\linewidth*0.9, -0.9) {$1$};
\node[stylered] (4) at (0.06666666666*\linewidth*0.9, -0.9) {$2$};
\node[stylered] (9) at (0.23333333333*\linewidth*0.9, -0.9) {$4$};

\node[styleblue] (11) at (0.6*\linewidth*0.9, -0.9) {$1$};
\node[stylered] (13) at (0.66666666666*\linewidth*0.9, -0.9) {$2$};
\node[stylered] (16) at (0.83333333333*\linewidth*0.9, -0.9) {$4$};

\node[styleblue] (17) at (0.43333333333*\linewidth*0.9, -1.9) {$5$};
\node[styleblue] (18) at (0.56666666666*\linewidth*0.9, -1.9) {$6$};
\node[stylered] (20) at (0.66666666666*\linewidth*0.9, -1.9) {$3$};
\node[stylered] (22) at (0.8*\linewidth*0.9, -1.9) {$7$};

\node[styleblue] (25) at (0.1*\linewidth*0.9, -2.9) {$5$};
\node[styleblue] (27) at (0.23333333333*\linewidth*0.9, -2.9) {$6$};
\node[stylered] (29) at (0.33333333333*\linewidth*0.9, -2.9) {$3$};
\node[stylered] (31) at (0.46666666666*\linewidth*0.9, -2.9) {$7$};

\node[font=\scriptsize] at (0*\linewidth*0.9, -0.5) {+3};
\node[font=\scriptsize] at (0.06666666666*\linewidth*0.9, -0.5) {-1};
\node[font=\scriptsize] at (0.23333333333*\linewidth*0.9, -0.5) {-2};

\node[font=\scriptsize] at (0.6*\linewidth*0.9, -0.5) {+2};
\node[font=\scriptsize] at (0.66666666666*\linewidth*0.9, -0.5) {-1};
\node[font=\scriptsize] at (0.83333333333*\linewidth*0.9, -0.5) {-1};

\node[font=\scriptsize] at (0.43333333333*\linewidth*0.9, -1.5) {+1};
\node[font=\scriptsize] at (0.56666666666*\linewidth*0.9, -1.5) {+1};
\node[font=\scriptsize] at (0.66666666666*\linewidth*0.9, -1.5) {-1};
\node[font=\scriptsize] at (0.8*\linewidth*0.9, -1.5) {-1};

\node[font=\scriptsize] at (0.1*\linewidth*0.9, -2.5) {+2};
\node[font=\scriptsize] at (0.23333333333*\linewidth*0.9, -2.5) {+1};
\node[font=\scriptsize] at (0.33333333333*\linewidth*0.9, -2.5) {-2};
\node[font=\scriptsize] at (0.46666666666*\linewidth*0.9, -2.5) {-1};

\node[stylepurple] (40) at (0*\linewidth*0.9, -4) { };
\node[font=\scriptsize] at (0*\linewidth*0.9, -3.77) {$1$};
\node[stylepurple] (41) at (0.06666*\linewidth*0.9, -4) {};
\node[stylepurple] (42) at (0.133333*\linewidth*0.9, -4) {};
\node[stylepurple] (43) at (0.233333*\linewidth*0.9, -4) {};
\node[stylepurple] (44) at (0.399999*\linewidth*0.9, -4) {};
\node[stylepurple] (45) at (0.466666*\linewidth*0.9, -4) {};
\node[stylepurple] (46) at (0.499999*\linewidth*0.9, -4) {};
\node[stylepurple] (47) at (0.599999*\linewidth*0.9, -4) {};
\node[stylepurple] (48) at (0.733333*\linewidth*0.9, -4) {};
\node[stylepurple] (49) at (0.766666*\linewidth*0.9, -4) {};
\node[stylepurple] (50) at (0.833333*\linewidth*0.9, -4) {};

\node[stylepurple] (100) at (0*\linewidth*0.9, -4.5) {};
\node[font=\scriptsize] at (0*\linewidth*0.9, -4.27) {$2$};
\node[stylepurple] (51) at (0.066666*\linewidth*0.9, -4.5) {};
\node[stylepurple] (52) at (0.133333*\linewidth*0.9, -4.5) {};
\node[stylepurple] (53) at (0.166666*\linewidth*0.9, -4.5) {};
\node[stylepurple] (54) at (0.399999*\linewidth*0.9, -4.5) {};
\node[stylepurple] (55) at (0.433333*\linewidth*0.9, -4.5) {};
\node[stylepurple] (56) at (0.466666*\linewidth*0.9, -4.5) {};
\node[stylepurple] (57) at (0.533333*\linewidth*0.9, -4.5) {};
\node[stylepurple] (58) at (0.666666*\linewidth*0.9, -4.5) {};
\node[stylepurple] (59) at (0.766666*\linewidth*0.9, -4.5) {};
\node[stylepurple] (60) at (0.699999*\linewidth*0.9, -4.5) {};
\node[stylepurple] (105) at (0.833333*\linewidth*0.9, -4.5) {};

\node[stylepurple] (101) at (0*\linewidth*0.9, -5) {};
\node[font=\scriptsize] at (0*\linewidth*0.9, -4.77) {$3$};
\node[stylepurple] (61) at (0.133333*\linewidth*0.9, -5) {};
\node[stylepurple] (62) at (0.166666*\linewidth*0.9, -5) {};
\node[stylepurple] (63) at (0.299999*\linewidth*0.9, -5) {};
\node[stylepurple] (64) at (0.333333*\linewidth*0.9, -5) {};
\node[stylepurple] (65) at (0.433333*\linewidth*0.9, -5) {};
\node[stylepurple] (66) at (0.466666*\linewidth*0.9, -5) {};
\node[stylepurple] (67) at (0.599999*\linewidth*0.9, -5) {};
\node[stylepurple] (68) at (0.666666*\linewidth*0.9, -5) {};
\node[stylepurple] (69) at (0.766666*\linewidth*0.9, -5) {};
\node[stylepurple] (106) at (0.833333*\linewidth*0.9, -5) {};

\node[stylepurple] (102) at (0*\linewidth*0.9, -5.5) {};
\node[stylepurple] (70) at (0.066666*\linewidth*0.9, -5.5) {};
\node[font=\scriptsize] at (0*\linewidth*0.9, -5.28) {$4$};
\node[stylepurple] (71) at (0.133333*\linewidth*0.9, -5.5) {};
\node[stylepurple] (72) at (0.233333*\linewidth*0.9, -5.5) {};
\node[stylepurple] (73) at (0.366666*\linewidth*0.9, -5.5) {};
\node[stylepurple] (74) at (0.399999*\linewidth*0.9, -5.5) {};
\node[stylepurple] (75) at (0.566666*\linewidth*0.9, -5.5) {};
\node[stylepurple] (76) at (0.733333*\linewidth*0.9, -5.5) {};
\node[stylepurple] (77) at (0.833333*\linewidth*0.9, -5.5) {};

\node[stylepurple] (103) at (0*\linewidth*0.9, -6) {};
\node[font=\scriptsize] at (0*\linewidth*0.9, -5.78) {$5$};
\node[stylepurple] (78) at (0.099999*\linewidth*0.9, -6) {};
\node[stylepurple] (79) at (0.299999*\linewidth*0.9, -6) {};
\node[stylepurple] (80) at (0.433333*\linewidth*0.9, -6) {};
\node[stylepurple] (81) at (0.533333*\linewidth*0.9, -6) {};
\node[stylepurple] (82) at (0.566666*\linewidth*0.9, -6) {};
\node[stylepurple] (83) at (0.833333*\linewidth*0.9, -6) {};

\node[stylepurple] (104) at (0*\linewidth*0.9, -6.5) {};
\node[font=\scriptsize] at (0*\linewidth*0.9, -6.27) {$6$};
\node[stylepurple] (84) at (0.166666*\linewidth*0.9, -6.5) {};
\node[stylepurple] (85) at (0.233333*\linewidth*0.9, -6.5) {};
\node[stylepurple] (86) at (0.299999*\linewidth*0.9, -6.5) {};
\node[stylepurple] (87) at (0.399999*\linewidth*0.9, -6.5) {};
\node[stylepurple] (88) at (0.433333*\linewidth*0.9, -6.5) {};
\node[stylepurple] (89) at (0.566666*\linewidth*0.9, -6.5) {};
\node[stylepurple] (90) at (0.699999*\linewidth*0.9, -6.5) {};
\node[stylepurple] (91) at (0.766666*\linewidth*0.9, -6.5) {};
\node[stylepurple] (107) at (0.833333*\linewidth*0.9, -6.5) {};

\node[stylepurple] (92) at (0*\linewidth*0.9, -7) {};
\node[font=\scriptsize] at (0*\linewidth*0.9, -6.77) {$7$};
\node[stylepurple] (93) at (0.066666*\linewidth*0.9, -7) {};
\node[stylepurple] (94) at (0.299999*\linewidth*0.9, -7) {};
\node[stylepurple] (95) at (0.333333*\linewidth*0.9, -7) {};
\node[stylepurple] (96) at (0.466666*\linewidth*0.9, -7) {};
\node[stylepurple] (97) at (0.799999*\linewidth*0.9, -7) {};
\node[stylepurple] (108) at (0.833333*\linewidth*0.9, -7) {};

\node[styleblack] (98) at (-2, -5.5) {$s$};
\node[styleblack] (99) at (\linewidth*0.85, -5.5) {$t$};

\path [->, ultra thick, out=340, in=50] (4) edge node[midway, right,font=\scriptsize] {$2c^{m}$} (42);
\path [->, ultra thick, out=90, in=270] (47) edge node[midway, right,font=\scriptsize] {$0$} (11);

\path [->, out=0, in=180] (40) edge (41);
\path [->, out=0, in=180] (41) edge (42);
\path [->, out=0, in=180] (42) edge (43);
\path [->, ultra thick, out=0, in=180] (43) edge node[midway, above,font=\scriptsize] {$5c^{p}_{1}$} (44);
\path [->, out=0, in=180] (44) edge (45);
\path [->, out=0, in=180] (45) edge (46);
\path [->, out=0, in=180] (46) edge (47);
\path [->, out=0, in=180] (47) edge (48);
\path [->, out=0, in=180] (48) edge (49);
\path [->, out=0, in=180] (49) edge (50);

\path [->, out=0, in=180] (100) edge (51);
\path [->, out=0, in=180] (51) edge (52);
\path [->, out=0, in=180] (52) edge (53);
\path [->, out=0, in=180] (53) edge (54);
\path [->, out=0, in=180] (54) edge (55);
\path [->, out=0, in=180] (55) edge (56);
\path [->, out=0, in=180] (56) edge (57);
\path [->, out=0, in=180] (57) edge (58);
\path [->, out=0, in=180] (58) edge (60);
\path [->, out=0, in=180] (60) edge (59);
\path [->, out=0, in=180] (59) edge (105);

\path [->, out=0, in=180] (101) edge (61);
\path [->, out=0, in=180] (61) edge (62);
\path [->, out=0, in=180] (62) edge (63);
\path [->, out=0, in=180] (63) edge (64);
\path [->, out=0, in=180] (64) edge (65);
\path [->, out=0, in=180] (65) edge (66);
\path [->, out=0, in=180] (66) edge (67);
\path [->, out=0, in=180] (67) edge (68);
\path [->, out=0, in=180] (68) edge (69);
\path [->, out=0, in=180] (69) edge (106);

\path [->, out=0, in=180] (102) edge (70);
\path [->, out=0, in=180] (70) edge (71);
\path [->, out=0, in=180] (71) edge (72);
\path [->, out=0, in=180] (72) edge (73);
\path [->, out=0, in=180] (73) edge (74);
\path [->, out=0, in=180] (74) edge (75);
\path [->, out=0, in=180] (75) edge (76);
\path [->, out=0, in=180] (76) edge (77);

\path [->, out=0, in=180] (103) edge (78);
\path [->, out=0, in=180] (78) edge (79);
\path [->, out=0, in=180] (79) edge (80);
\path [->, out=0, in=180] (80) edge (81);
\path [->, out=0, in=180] (81) edge (82);
\path [->, out=0, in=180] (82) edge (83);

\path [->, out=0, in=180] (104) edge (84);
\path [->, out=0, in=180] (84) edge (85);
\path [->, out=0, in=180] (85) edge (86);
\path [->, out=0, in=180] (86) edge (87);
\path [->, out=0, in=180] (87) edge (88);
\path [->, out=0, in=180] (88) edge (89);
\path [->, out=0, in=180] (89) edge (90);
\path [->, out=0, in=180] (90) edge (91);
\path [->, out=0, in=180] (91) edge (107);

\path [->, out=0, in=180] (92) edge (93);
\path [->, out=0, in=180] (93) edge (94);
\path [->, out=0, in=180] (94) edge (95);
\path [->, out=0, in=180] (95) edge (96);
\path [->, out=0, in=180] (96) edge (97);
\path [->, out=0, in=180] (97) edge (108);

\path [->, ultra thick] (98) edge node[midway, above,font=\scriptsize] {$c^{f}$} (40);
\path [->] (98) edge (100);
\path [->] (98) edge (101);
\path [->] (98) edge (102);
\path [->] (98) edge (103);
\path [->] (98) edge (104);
\path [->] (98) edge (92);

\path [->, ultra thick] (50) edge node[midway, above,font=\scriptsize] {$0$} (99);
\path [->] (105) edge (99);
\path [->] (106) edge (99);
\path [->] (77) edge (99);
\path [->] (83) edge (99);
\path [->] (107) edge (99);
\path [->] (108) edge (99);

\end{tikzpicture}
    \caption{Time-space network constructed for the four bus runs in Figure~\ref{fig:bus_run}.}
    \label{fig:timespace}
\end{figure}

\begin{algorithm}[H]
\small
\caption{Construct time-space network}
\label{alg:timespace}
\begin{algorithmic}
\State \textbf{Input:} Bus runs \(B\), stops \(S\), and travel times \(t_{ss'}\).
\State \textbf{Output:} Time-space network \((V,E)\).
\State \(V \gets \{\mathrm{source},\mathrm{sink}\},\quad E \gets \emptyset\)

\For{each \(b\in B\)}
    \State add a requester at the first stop and a releaser at the last stop of \(b\)
    \State add requester nodes at \(s\in S_b\) with demand
    \(d_{b,s}-d_{b,s^-}\) if \(d_{b,s}>d_{b,s^-}\)
    \State add releaser nodes at \(s\in S_b\) with supply
    \(d_{b,s}-d_{b,s^+}\) if \(d_{b,s}>d_{b,s^+}\)
\EndFor

\For{each requester \((q,s,t)\) and \(s'\in S\)}
    \State add \((w,s',t-t_{s's})\) to \(V\), if absent
    \State add \((w,s',t-t_{s's})\rightarrow(q,s,t)\) to \(E\)
    with cost \(c^{\mathrm m}t_{s's}\)
\EndFor

\For{each releaser \((r,s,t)\) and \(s'\in S\)}
    \State add \((w,s',t+t_{ss'})\) to \(V\), if absent
    \State add \((r,s,t)\rightarrow(w,s',t+t_{ss'})\) to \(E\)
    with cost \(c^{\mathrm m}t_{ss'}\)
\EndFor

\For{each \(s\in S\)}
    \State sort standby nodes at \(s\) by time
    \State add parking edges between consecutive nodes with cost
    \(c^{\mathrm p}_{s}\Delta t\)
    \State add a source edge to the earliest standby node with cost \(c^{\mathrm f}\)
    \State add a zero-cost edge from the latest standby node to the sink
\EndFor

\State \Return \((V,E)\)
\end{algorithmic}
\end{algorithm}

A minimum-cost flow in the constructed time-space network defines an optimal solution to the pod routing problem. The total flow leaving the source equals the number of pods deployed in the solution. Individual pod routes are recovered by decomposing the flow into unit flows and matching requester and releaser nodes associated with the same bus run.

\section{Solution Approach}\label{sec:SolutionApproach}

The minimum-cost flow formulation on the time-space network introduced in Section~\ref{sec:Model} becomes computationally expensive for large transit systems. To avoid solving the minimum-cost flow problem on the full network, we propose a two-stage optimization framework called Hierarchical Pod Routing (HPR). The first stage decomposes bus runs into pod routes and assigns the resulting routes to pods using a minimum disjoint path cover on a directed acyclic graph. The second stage selects a waiting location between each pair of consecutive pod routes to minimize deadheading and parking costs. The following subsections describe the two stages of HPR.

\subsubsection*{Stage 1: Pod Assignment}

The first stage of HPR constructs feasible sequences of service assignments for individual pods.This stage is formulated as a minimum disjoint path cover on a directed acyclic graph (DAG), where each path represents the sequence of assignments served by one pod.

A bus run may require multiple pods, each of which may serve only a portion of the run.For a given bus run, a pod route is defined as a contiguous sequence of stops served by one pod. When a pod leaves a bus run before its final stop, its passengers must move to the pods that continue serving the run.Minimizing the number of pod routes therefore limits the number of occasions on which passengers must move between pods. Figure~\ref{fig:decomposition} illustrates a decomposition of the bus runs in Figure~\ref{fig:bus_run} into pod routes.

\begin{figure}[tb]
    \centering
    \begin{tikzpicture}
\draw[->] (-0.5,0) -- (0.8*\linewidth,0) node[right]{ };

\foreach \x/\time in {0/13:00, 0.06666666666/13:02, 0.1/13:03, 0.16666666666/13:05, 
    0.23333333333/13:07, 0.33333333333/13:10, 0.43333333333/13:13, 
    0.46666666666/13:14, 0.56666666666/13:17,0.6/13:18, 0.66666666666/13:20, 
    0.76666666666/13:23, 0.8/13:24, 0.83333333333/13:25} 
{
    \draw (\x*\linewidth*0.9, -0.1) -- (\x*\linewidth*0.9, 0.1);
    \draw (\x*\linewidth*0.9 + 0.013*\linewidth, 0.5) node[above, rotate=90, font=\scriptsize] {\time};
}

\node[styleblack] (1) at (0*\linewidth*0.9, -1) {$1$};
\node[styleblack] (2) at (0*\linewidth*0.9, -1.7) {$1$};
\node[styleblack] (3) at (0*\linewidth*0.9, -2.4) {$1$};

\node[styleblack] (4) at (0.06666666666*\linewidth*0.9, -1) {$2$};
\node[styleblack] (5) at (0.06666666666*\linewidth*0.9, -1.7) {$2$};
\node[styleblack] (6) at (0.06666666666*\linewidth*0.9, -2.4) {$2$};

\node[styleblack] (7) at (0.16666666666*\linewidth*0.9, -1) {$3$};
\node[styleblack] (8) at (0.16666666666*\linewidth*0.9, -1.7) {$3$};

\node[styleblack] (9) at (0.23333333333*\linewidth*0.9, -1) {$4$};
\node[styleblack] (10) at (0.23333333333*\linewidth*0.9, -1.7) {$4$};

\node[styleblack] (11) at (0.6*\linewidth*0.9, -1) {$1$};
\node[styleblack] (12) at (0.6*\linewidth*0.9, -1.7) {$1$};

\node[styleblack] (13) at (0.66666666666*\linewidth*0.9, -1) {$2$};
\node[styleblack] (14) at (0.66666666666*\linewidth*0.9, -1.7) {$2$};

\node[styleblack] (15) at (0.76666666666*\linewidth*0.9, -1) {$3$};
\node[styleblack] (16) at (0.83333333333*\linewidth*0.9, -1) {$4$};

\node[styleblack] (17) at (0.43333333333*\linewidth*0.9, -3) {$5$};
\node[styleblack] (18) at (0.56666666666*\linewidth*0.9, -3) {$6$};
\node[styleblack] (20) at (0.66666666666*\linewidth*0.9, -3) {$3$};
\node[styleblack] (22) at (0.8*\linewidth*0.9, -3) {$7$};

\node[styleblack] (19) at (0.56666666666*\linewidth*0.9, -3.7) {$6$};
\node[styleblack] (21) at (0.66666666666*\linewidth*0.9, -3.7) {$3$};

\node[styleblack] (23) at (0.23333333333*\linewidth*0.9, -4) {$6$};
\node[styleblack] (24) at (0.33333333333*\linewidth*0.9, -4) {$3$};

\node[styleblack] (25) at (0.1*\linewidth*0.9, -5.4) {$5$};
\node[styleblack] (27) at (0.23333333333*\linewidth*0.9, -5.4) {$6$};
\node[styleblack] (29) at (0.33333333333*\linewidth*0.9, -5.4) {$3$};
\node[styleblack] (31) at (0.46666666666*\linewidth*0.9, -5.4) {$7$};

\node[styleblack] (26) at (0.1*\linewidth*0.9, -4.7) {$5$};
\node[styleblack] (28) at (0.23333333333*\linewidth*0.9, -4.7) {$6$};
\node[styleblack] (30) at (0.33333333333*\linewidth*0.9, -4.7) {$3$};

\path [->, thick, black] (1) edge (4);
\path [->, thick, black] (4) edge (7);
\path [->, thick, black] (7) edge (9);
\path [->, thick, black] (2) edge (5);
\path [->, thick, black] (5) edge (8);
\path [->, thick, black] (8) edge (10);
\path [->, thick, black] (3) edge (6);
\path [->, thick, black, out=0, in=180] (11) edge (13);
\path [->, thick, black,out=0, in=180] (13) edge (15);
\path [->, thick, black, out=0, in=180] (15) edge (16);
\path [->, thick, black, out=0, in=180] (12) edge (14);
\path [->, thick, black, out=0, in=180] (17) edge (18);
\path [->, thick, black, out=0, in=180] (18) edge (20);
\path [->, thick, black, out=0, in=180] (20) edge (22);
\path [->, thick, black, out=0, in=180] (19) edge (21);
\path [->, thick, black, out=0, in=180] (23) edge (24);
\path [->, thick, black, out=0, in=180] (25) edge (27);
\path [->, thick, black, out=0, in=180] (27) edge (29);
\path [->, thick, black, out=0, in=180] (29) edge (31);
\path [->, thick, black, out=0, in=180] (26) edge (28);
\path [->, thick, black, out=0, in=180] (28) edge (30);
\end{tikzpicture}
    \caption{Decomposition of the bus runs shown in Figure~\ref{fig:bus_run} into individual pod routes. Each route represents the in-service journey of a single pod.}
    \label{fig:decomposition}
\end{figure}

Minimizing the number of pod routes therefore limits the number of occasions on which passengers must move between pods.Starting at the first stop with positive residual demand, the algorithm constructs a pod route by extending it through consecutive stops with positive residual demand. The residual demand at each stop included in the route is then reduced by one. The procedure continues until the residual demand is zero at every stop.

\begin{algorithm}[tb]
\caption{decompose\_bus\_run}
\label{alg:decompose}
\begin{algorithmic}
\small
    \State \textbf{Input:} Bus run \(b\) with ordered stations \([s_1, \dots, s_n]\), demands \(\mathbf{d} = (d_{s_1}, \dots, d_{s_n})\)
    \State \textbf{Output:} Set of pod routes $P$
    \State $P \gets \emptyset$.
    \While{any $d_s > 0$} 
        \State Find first station $s_i$ with $d_{s_i} > 0$.
        \State Start new route $p$ from $s_i$.
        \State Extend $p$ by adding consecutive $s_j$ with $d_{s_j} > 0$.
        \State Decrement $d_s$ by 1 for all $s \in p$.
        \State Add $p$ to $P$.
    \EndWhile
    \State \Return $P$
\end{algorithmic}
\end{algorithm}

Theorem~\ref{Thrm:1} shows that Algorithm~\ref{alg:decompose} produces the minimum possible number of pod routes. The proof is provided in Appendix~\ref{app:proof_decomposition}.

\begin{theorem}
\label{Thrm:1}
For a given bus run, Algorithm~\ref{alg:decompose} constructs the minimum number of pod routes required to serve its passenger demand.
\end{theorem}

Let \(P\) denote the set of pod routes produced by Algorithm~\ref{alg:decompose} for all bus runs. A directed graph is constructed from the pod routes in \(P\). For each pod route \(p \in P\), the graph contains a requester vertex \((q,s_p^{\mathrm{s}},t_p^{\mathrm{s}})\) and a releaser vertex \((r,s_p^{\mathrm{e}},t_p^{\mathrm{e}})\), representing the start and end of the route, respectively. A directed arc connects the requester vertex to the corresponding releaser vertex of each pod route. For two distinct pod routes \(p,p' \in P\), a directed edge is added from the releaser vertex of \(p\) to the requester vertex of \(p\) whenever
\begin{equation}
t_{s_{p}^{\mathrm{e}}s_{p'}^{\mathrm{s}}}
\le
t_{p'}^{\mathrm{s}}-t_{p}^{\mathrm{e}}
\end{equation}
This condition ensures that a pod completing route \(p\) can deadhead to the starting stop of route \(p'\) before its scheduled start time. Figure~\ref{fig:H} illustrates the directed graph constructed from the pod routes in Figure~\ref{fig:decomposition}. Algorithm~\ref{alg:bipartite} summarizes the construction of the bipartite representation from the set of bus runs. Because every arc follows the temporal order of the pod routes, time strictly increases along each directed path. Therefore, the constructed graph is acyclic.

\begin{algorithm}[tb]
\caption{construct\_bipartite\_graph}
\label{alg:bipartite}
\begin{algorithmic}
\small
    \State \textbf{Input:} Set of bus runs \(B\)
    \State \textbf{Output:} Bipartite graph \((Q,R,E)\)

    \State \(P \gets \bigcup_{b \in B}\texttt{decompose\_bus\_run}(b)\)
    \State \(Q \gets \{(q,s_p^{\mathrm{s}},t_p^{\mathrm{s}}): p \in P\}\)
    \State \(R \gets \{(r,s_p^{\mathrm{e}},t_p^{\mathrm{e}}): p \in P\}\)
    \State \(E \gets \{((r,s_p^{\mathrm{e}},t_p^{\mathrm{e}}),(q,s_{p'}^{\mathrm{s}},t_{p'}^{\mathrm{s}})) : p,p' \in P,\;
    t_{s_p^{\mathrm{e}}s_{p'}^{\mathrm{s}}} \leq t_{p'}^{\mathrm{s}}-t_p^{\mathrm{e}}\}\)
    \State \Return \((Q,R,E)\)
\end{algorithmic}
\end{algorithm}
 
\begin{figure}[tb]
    \centering
    \begin{tikzpicture}
\draw[->] (-0.5,0) -- (0.8*\linewidth,0) node[right]{ };

\foreach \x/\time in {0/13:00, 0.06666666666/13:02, 0.1/13:03, 0.16666666666/13:05, 
    0.23333333333/13:07, 0.33333333333/13:10, 0.43333333333/13:13, 
    0.46666666666/13:14, 0.56666666666/13:17,0.6/13:18, 0.66666666666/13:20, 
    0.76666666666/13:23, 0.8/13:24, 0.83333333333/13:25} 
{
    \draw (\x*\linewidth*0.9, -0.1) -- (\x*\linewidth*0.9, 0.1);
    \draw (\x*\linewidth*0.9 + 0.013*\linewidth, 0.5) node[above, rotate=90, font=\scriptsize] {\time};
}

\node[styleblue] (1) at (0*\linewidth*0.9, -1) {$1$};
\node[styleblue] (2) at (0*\linewidth*0.9, -2) {$1$};
\node[styleblue] (3) at (0*\linewidth*0.9, -3) {$1$};

\node[stylered] (6) at (0.06666666666*\linewidth*0.9, -3) {$2$};

\node[stylered] (9) at (0.23333333333*\linewidth*0.9, -1) {$4$};
\node[stylered] (10) at (0.23333333333*\linewidth*0.9, -2) {$4$};

\node[styleblue] (11) at (0.6*\linewidth*0.9, -1) {$1$};
\node[styleblue] (12) at (0.6*\linewidth*0.9, -2) {$1$};

\node[stylered] (14) at (0.66666666666*\linewidth*0.9, -2) {$2$};

\node[stylered] (16) at (0.83333333333*\linewidth*0.9, -1) {$4$};

\node[styleblue] (17) at (0.43333333333*\linewidth*0.9, -3) {$5$};

\node[stylered] (22) at (0.8*\linewidth*0.9, -3) {$7$};

\node[styleblue] (19) at (0.56666666666*\linewidth*0.9, -4) {$6$};
\node[stylered] (21) at (0.66666666666*\linewidth*0.9, -4) {$3$};

\node[styleblue] (23) at (0.23333333333*\linewidth*0.9, -5) {$6$};
\node[stylered] (24) at (0.33333333333*\linewidth*0.9, -5) {$3$};

\node[styleblue] (25) at (0.1*\linewidth*0.9, -7) {$5$};

\node[stylered] (31) at (0.46666666666*\linewidth*0.9, -7) {$7$};

\node[styleblue] (26) at (0.1*\linewidth*0.9, -6) {$5$};
\node[stylered] (30) at (0.33333333333*\linewidth*0.9, -6) {$3$};

\path [->,ultra thick, black] (1) edge (9);
\path [->,ultra thick, black] (2) edge (10);
\path [->,ultra thick, black] (3) edge (6);
\path [->, gray, out=0, in=180] (6) edge (11);
\path [->, gray, out=0, in=180] (6) edge (12);
\path [->, gray, out=0, in=180] (6) edge (23);
\path [->, gray, out=0, in=200] (6) edge (19);
\path [->,ultra thick, black, out=0, in=180] (11) edge (16);
\path [->,ultra thick, black, out=0, in=180] (12) edge (14);
\path [->, gray, out=-60, in=200] (9) edge (19);
\path [->, gray, out=-45, in=200] (10) edge (19);

\path [->, gray, out=0, in=200] (30) edge (19);
\path [->, gray, out=0, in=180] (24) edge (11);
\path [->, gray, out=0, in=180] (30) edge (12);
\path [->, gray, out=0, in=180] (10) edge (11);
\path [->, gray, out=0, in=180] (9) edge (12);
\path [->,ultra thick, black, out=0, in=180] (9) edge (11);
\path [->,ultra thick, black, out=0, in=180] (10) edge (12);
\path [->,ultra thick, black, out=0, in=180] (6) edge (17);
\path [->,ultra thick, black, out=0, in=180] (17) edge (22);
\path [->,ultra thick, black, out=0, in=180] (19) edge (21);
\path [->,ultra thick, black, out=0, in=180] (23) edge (24);
\path [->,ultra thick, black, out=0, in=200] (24) edge (19);
\path [->,ultra thick, black, out=0, in=180] (25) edge (31);
\path [->,ultra thick, black, out=0, in=180] (26) edge (30);
\end{tikzpicture}
    \caption{Directed graph constructed from the bus transit schedule in Figure~\ref{fig:decomposition}. Edges indicate feasible pod transitions. A minimum disjoint path cover is emphasized with thick black lines.}
    \label{fig:H}
\end{figure}

A directed path in the constructed DAG represents a feasible sequence of pod routes assigned to the same pod. A vertex-disjoint path cover of the DAG therefore defines a feasible pod assignment that covers all pod routes. The minimum vertex-disjoint path cover therefore gives the minimum number of pods required to serve all pod routes.

The minimum disjoint path-cover problem on the constructed DAG can be solved through an equivalent maximum bipartite matching problem. A bipartite graph is constructed by creating one partition for requester vertices and another partition for releaser vertices. For every directed edge from a releaser vertex to a requester vertex in the DAG, an edge is added between the corresponding vertices in the bipartite graph.

The minimum disjoint path-cover problem on the constructed DAG can be solved through an equivalent maximum bipartite matching problem. Maximizing the number of matched edges minimizes the number of disjoint paths and therefore the number of pods required to cover all pod routes. The resulting disjoint paths define the sequence of pod routes assigned to each pod. Figure~\ref{fig:reducedH} illustrates the bipartite graph constructed from the DAG in Figure~\ref{fig:H}. The bold edges indicate a maximum matching and the resulting sequences of pod routes.

\begin{figure}[tb]
    \centering
    \begin{tikzpicture}
\draw[->] (-0.5,0) -- (0.8*\linewidth,0) node[right]{ };

\foreach \x/\time in {0/13:00, 0.06666666666/13:02, 0.1/13:03, 0.16666666666/13:05, 
    0.23333333333/13:07, 0.33333333333/13:10, 0.43333333333/13:13, 
    0.46666666666/13:14, 0.56666666666/13:17,0.6/13:18, 0.66666666666/13:20, 
    0.76666666666/13:23, 0.8/13:24, 0.83333333333/13:25} 
{
    \draw (\x*\linewidth*0.9, -0.1) -- (\x*\linewidth*0.9, 0.1);
    \draw (\x*\linewidth*0.9 + 0.013*\linewidth, 0.5) node[above, rotate=90, font=\scriptsize] {\time};
}

\node[styleblue] (1) at (0*\linewidth*0.9, -1) {$1$};
\node[styleblue] (2) at (0*\linewidth*0.9, -2) {$1$};
\node[styleblue] (3) at (0*\linewidth*0.9, -3) {$1$};
\node[stylered] (6) at (0.06666666666*\linewidth*0.9, -3) {$2$};
\node[stylered] (9) at (0.23333333333*\linewidth*0.9, -1) {$4$};
\node[stylered] (10) at (0.23333333333*\linewidth*0.9, -2) {$4$};
\node[styleblue] (11) at (0.6*\linewidth*0.9, -1) {$1$};
\node[styleblue] (12) at (0.6*\linewidth*0.9, -2) {$1$};
\node[stylered] (14) at (0.66666666666*\linewidth*0.9, -2) {$2$};
\node[stylered] (16) at (0.83333333333*\linewidth*0.9, -1) {$4$};
\node[styleblue] (17) at (0.43333333333*\linewidth*0.9, -3) {$5$};
\node[stylered] (22) at (0.8*\linewidth*0.9, -3) {$7$};
\node[styleblue] (19) at (0.56666666666*\linewidth*0.9, -4) {$6$};
\node[stylered] (21) at (0.66666666666*\linewidth*0.9, -4) {$3$};
\node[styleblue] (23) at (0.23333333333*\linewidth*0.9, -5) {$6$};
\node[stylered] (24) at (0.33333333333*\linewidth*0.9, -5) {$3$};
\node[styleblue] (25) at (0.1*\linewidth*0.9, -7) {$5$};
\node[stylered] (31) at (0.46666666666*\linewidth*0.9, -7) {$7$};
\node[styleblue] (26) at (0.1*\linewidth*0.9, -6) {$5$};
\node[stylered] (30) at (0.33333333333*\linewidth*0.9, -6) {$3$};

\path [-,thick, gray, out=0, in=180] (6) edge (11);
\path [-,thick, gray, out=0, in=180] (6) edge (12);
\path [-,thick, gray, out=0, in=180] (6) edge (23);
\path [-,thick, gray, out=0, in=200] (6) edge (19);
\path [-,thick, gray, out=-60, in=200] (9) edge (19);
\path [-,thick, gray, out=-45, in=200] (10) edge (19);
\path [-,thick, gray, out=0, in=200] (30) edge (19);
\path [-,thick, gray, out=0, in=180] (24) edge (11);
\path [-,thick, gray, out=0, in=180] (30) edge (12);
\path [-,thick, gray, out=0, in=180] (10) edge (11);
\path [-,thick, gray, out=0, in=180] (9) edge (12);
\path [-,ultra thick, black, out=0, in=180] (9) edge (11);
\path [-,ultra thick, black, out=0, in=180] (10) edge (12);
\path [-,ultra thick, black, out=0, in=180] (6) edge (17);
\path [-,ultra thick, black, out=0, in=200] (24) edge (19);
\end{tikzpicture}
\caption{Bipartite graph associated with the DAG in Figure~\ref{fig:H}. Edges represent feasible transitions from releaser vertices to requester vertices. Bold edges indicate a maximum matching.}
    \label{fig:reducedH}
\end{figure}

\subsubsection*{Stage 2: Between-Service Routing}
\begin{subequations}\label{eq:parking_stop_selection}

The second stage selects a waiting location between each pair of consecutive pod routes assigned to the same pod. Consider a matched pair consisting of the releaser vertex \((r,s_p^{\mathrm{e}},t_p^{\mathrm{e}})\) of pod route \(p\) and the requester vertex \((q,s_{p'}^{\mathrm{s}},t_{p'}^{\mathrm{s}})\) of pod route \(p'\). Let \(x_s\) be a binary decision variable equal to 1 if stop \(s\in S\) is selected as the waiting location between pod routes \(p\) and \(p'\), and 0 otherwise. The waiting-location selection problem is formulated as follows:

\begin{align}
\min_{x_s} \quad
& \sum_{s \in S}
c^{\mathrm{m}}
\left(t_{s_{p}^e,s}+t_{s,s_{p'}^s}\right)x_s
\nonumber\\
& \quad +
\sum_{s \in S}
c^{\mathrm{p}}_{s}
\left(t_{p'}^s-t_{p}^e
-t_{s_{p}^e,s}-t_{s,s_{p'}^s}\right)x_s
\label{eq:parking_obj} \\
\text{s.t.} \quad
& \sum_{s \in S} x_s = 1,
\label{eq:parking_select_one} \\
& \left(t_{s_{p}^e,s}+t_{s,s_{p'}^s}\right)x_s
\le
\left(t_{p'}^s-t_{p}^e\right)x_s,
\quad s \in S,
\label{eq:parking_feasibility} \\
& x_s \in \{0,1\},
\quad s \in S.
\label{eq:parking_binary}
\end{align}
\end{subequations}

Although \eqref{eq:parking_stop_selection} is formulated as a binary optimization problem, solving it with an optimization solver is unnecessary. For each matched releaser--requester pair, the cost of selecting stop \(s\) can be evaluated independently for every \(s\in S\).The optimal solution is obtained by evaluating all stops that satisfy the travel-time constraint~\eqref{eq:parking_feasibility} and selecting the one with the lowest cost. Because the matching contains at most \(|P|\) edges, the total complexity of the second stage is \(O(|P||S|)\).

\section{Computational Experiments}\label{sec:CompExp}
This section describes the benchmark instances and computational setup used to compare HPR with the exact minimum-cost flow formulation. Section~\ref{sec:instances} explains the construction of the benchmark instances from static and real-time GTFS data. Section~\ref{sec:computational_setup} presents the cost parameters, implementation details, and computational environment used in the experiments.

\subsection{Benchmark Instances}\label{sec:instances}
Benchmark instances are constructed using publicly available static and real-time General Transit Feed Specification (GTFS) data from the New York City Metropolitan Transportation Authority (MTA). \footnote{\url{https://www.mta.info/developers}}\footnote{The data are available at \url{https://github.com/tinaradvand/modular-pod-routing-benchmarks}.} The static GTFS data provide route information, stop sequences, and scheduled service times, whereas the real-time GTFS data are used to estimate onboard passenger volumes and the corresponding pod requirements along each bus run.The initial dataset contains 41 local bus routes operating within Manhattan and $24,693$ scheduled weekday bus runs. Figure \ref{fig:active_trips} shows the number of active bus runs throughout a service day.

Real-time GTFS data were collected to estimate onboard passenger volumes along bus runs. API requests were submitted at one-minute intervals over 10 weekdays in February 2025. Each API response provided a snapshot of the buses operating in the system at that time. Onboard passenger counts were reported for only a subset of the real-time observations.

Each real-time GTFS observation was matched to the corresponding scheduled bus run and location along its stop sequence using the static GTFS data. Because real-time observations are frequently recorded while buses are traveling between stops, an observed onboard passenger volume was assigned as a lower bound to both adjacent stops. Observations from all collection days were aggregated, and the passenger volume for each bus run and stop was estimated as the maximum passenger count associated with that location.For a stop without an associated passenger observation, the passenger volume was estimated using the maximum count at the nearest preceding or following stop with available data. Overall, $40.1\%$ of the bus runs did not have any passenger count data and were excluded from the analysis. Figure~\ref{fig:occupancy_trend} shows the average onboard passenger volume throughout the service day, including both observed and estimated counts.

\begin{figure}[tb]
    \centering
    \begin{subfigure}[b]{0.48\linewidth}
        \includegraphics[width=\linewidth]{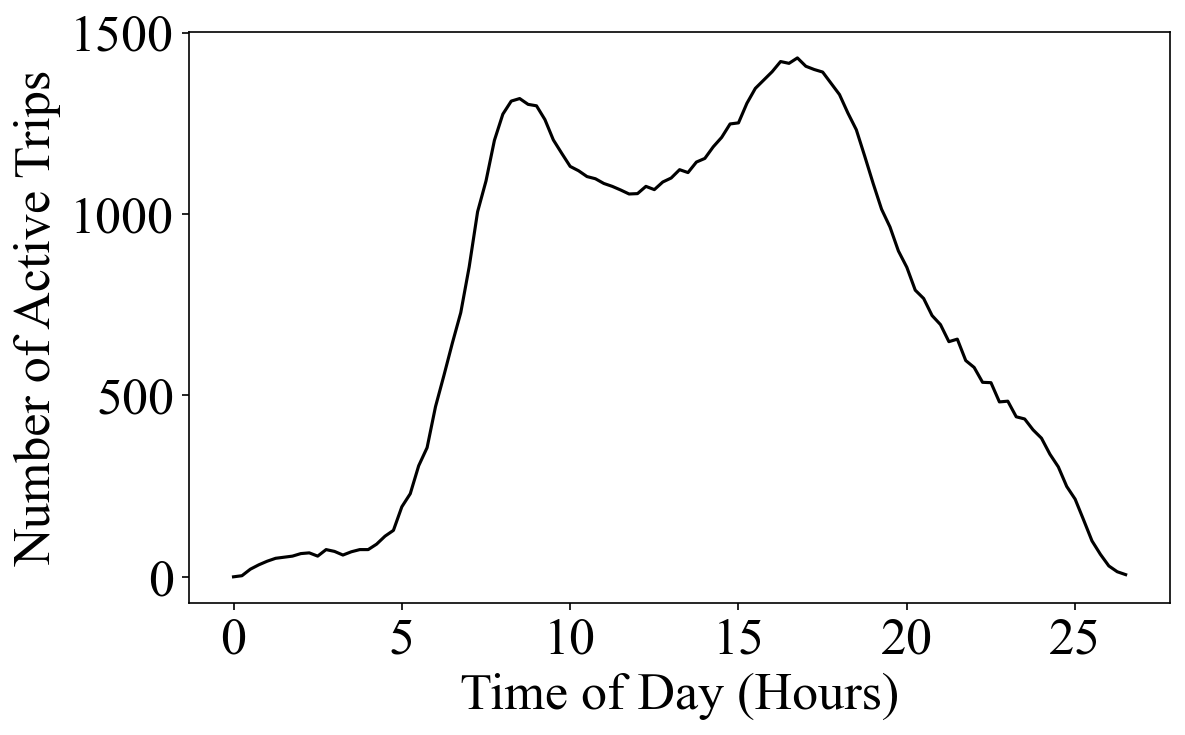}
        \caption{ }
        \label{fig:active_trips}
    \end{subfigure}
    \vspace{1em}  
    \begin{subfigure}[b]{0.48\linewidth}
        \includegraphics[width=\linewidth]{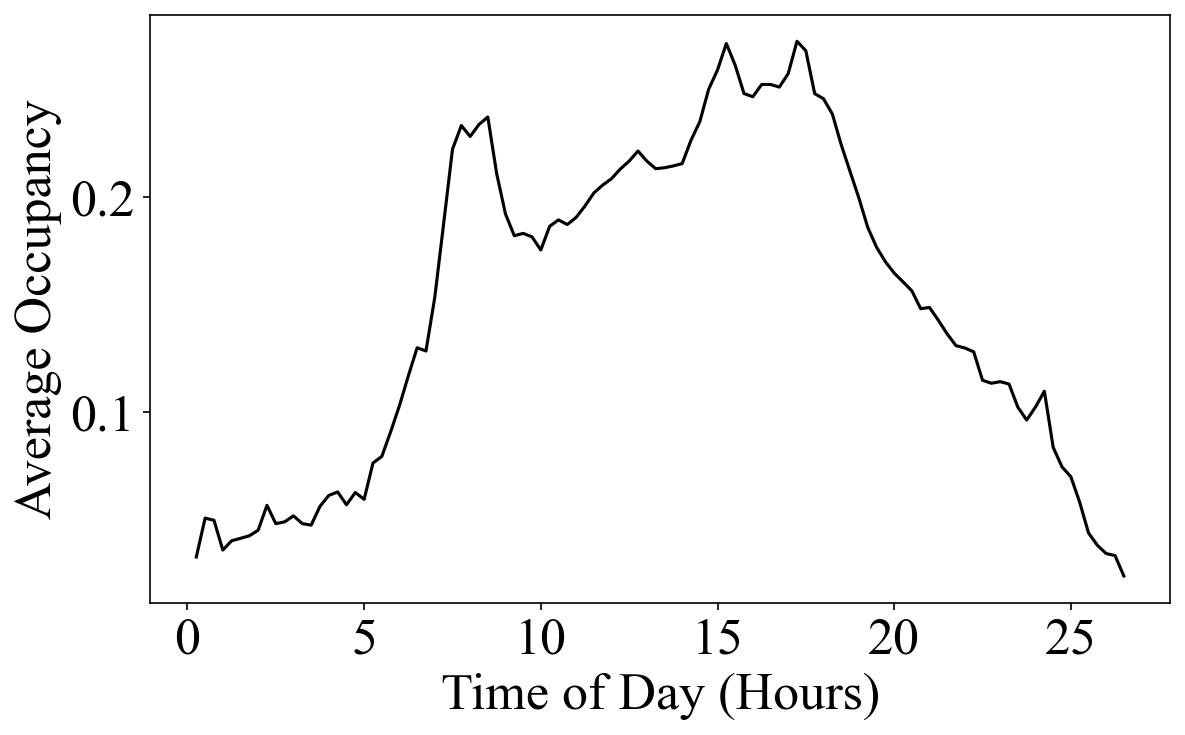}
        \caption{ }
        \label{fig:occupancy_trend}
    \end{subfigure}
    \caption{(a) Number of active bus runs throughout a weekday day. (b) Average bus occupancy over the day, combining real-time observations and estimated values. }
    \label{fig:bus_metrics}
\end{figure}

The estimated passenger counts were converted into pod requirements at each stop along each bus run. A capacity of 16 passengers per pod was assumed based on the modular vehicle design reported by \cite{tecnobus2025}. Because every scheduled bus run is served by at least one pod, the pod requirement for run \(b\) at stop \(s\) was calculated as \begin{equation} d_{b,s} = \max\left\{ 1, \left\lceil \frac{\hat n_{b,s}}{16} \right\rceil \right\}, \end{equation} where \(\hat n_{b,s}\) denotes the estimated onboard passenger count.

Each benchmark instance comprises a set of scheduled bus runs to be served by the modular pod system. The benchmark dataset is constructed from 36 local bus routes operating within Manhattan. For each route, three route-level instances are generated: one for each travel direction and one combining bus runs from both directions, yielding 108 route-level instances. Larger cumulative instances are then constructed by sequentially adding routes, beginning with instance 109. Figure~\ref{fig:benchmark_instances} reports the number of bus runs and distinct stops for the benchmark instances included in the computational analysis.

\begin{figure}[tb]
    \centering

    \begin{subfigure}[t]{0.49\textwidth}
        \centering
        \includegraphics[width=\linewidth]
        {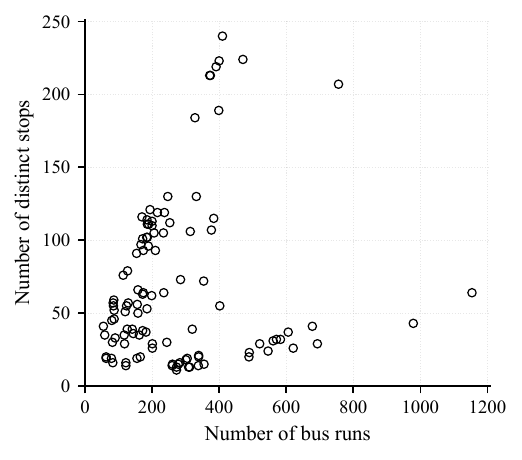}
        \caption{Route-level instances.}
        \label{fig:route_level_instances}
    \end{subfigure}
    \hfill
    \begin{subfigure}[t]{0.49\textwidth}
        \centering
        \includegraphics[width=\linewidth]
        {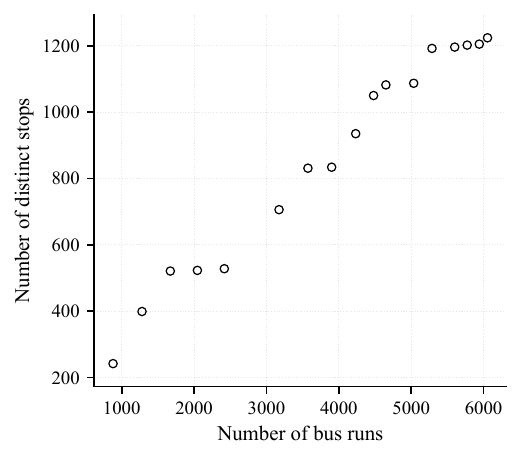}
        \caption{Cumulative multi-route instances.}
        \label{fig:cumulative_instances}
    \end{subfigure}

    \caption{Number of bus runs and distinct stops in the benchmark instances included in the computational analysis.}
    \label{fig:benchmark_instances}
\end{figure}

\subsection{Experimental and Computational Setup}
\label{sec:computational_setup}

The baseline cost parameters are selected to reflect daily pod operating costs.
The fixed cost is set to \(c^{\mathrm{f}}=13.70\) USD per pod per service day. This value is obtained by amortizing an assumed purchase price of 40,000 USD over an eight-year service life and is consistent with the pod capital-cost assumptions used by \citet{cheng2024autonomous}. The deadheading cost is set to \(c^{\mathrm{m}}=0.03\) USD per minute, based on an energy consumption of 1 kWh per mile, an average travel speed of 15 miles per hour, and an electricity price of 0.13 USD per kWh. The parking cost \(c^{\mathrm{p}}_s\) is sampled uniformly from \(\{0.01,0.02,\ldots,0.08\}\) USD per minute to represent variation in waiting costs across stops. The parking cost assigned to each stop was generated once and then kept fixed across all experiments.

The exact minimum-cost flow formulation is implemented in Python 3.11.2 using gurobipy 12.0.2 and Gurobi Optimizer 12.0.2. The experiments are conducted on a Google Cloud \texttt{c2d-highmem-112} virtual machine with 112 vCPUs and 896 GB of memory. The machine runs Debian GNU/Linux 12 on an \texttt{x86\_64} architecture. The model is solved using the barrier method with 112 threads and crossover disabled.

The maximum bipartite matching problem in the first stage of HPR is solved using a GPU-based push--relabel algorithm with global relabeling \citep{6687335}. The implementation is adapted from the CUDA code made publicly available by the authors.\footnote{\url{https://github.com/nagendervss/BipartiteMatchingCUDA}} HPR is executed on a Google Cloud \texttt{g2-standard-16} virtual machine with 16 vCPUs, 64 GB of memory, and one NVIDIA L4 GPU. The second-stage waiting-location problems are solved by evaluating all feasible stops for each pair of consecutive pod routes.

\section{Results and Discussion}\label{sec:Results}

The solution quality of HPR is evaluated by comparing its objective value with the optimal objective value obtained from the minimum-cost flow (MCF) formulation. the optimality gap is more variable among smaller instances but stabilizes at approximately $7--8\%$ for the largest benchmark instances. This pattern indicates that the relative loss in solution quality does not systematically increase with the number of bus runs. Thus, HPR maintains consistent solution quality across the largest benchmark instances.

\begin{figure}[tb]
    \centering
    \includegraphics[width=0.8\linewidth]{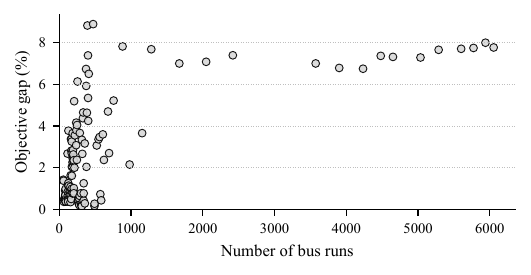}
    \caption{Objective gap between the HPR and MCF solutions across the instances.}
    \label{fig:gap}
\end{figure}

Figure~\ref{fig:computational_performance} compares the runtime and memory requirements of MCF and HPR across the benchmark instances. As the number of bus runs increases, the runtime advantage of HPR becomes increasingly pronounced.For the largest benchmark instances, the runtime of MCF exceeds that of HPR by more than four orders of magnitude, while its memory use is approximately one order of magnitude greater. These results show that HPR can solve large benchmark instances with substantially lower runtime and memory requirements than MCF.
\begin{figure}[tb]
    \centering

    \begin{subfigure}[t]{0.49\linewidth}
        \centering
        \includegraphics[width=\linewidth]{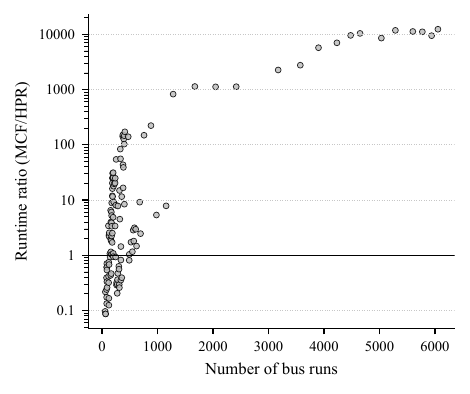}
        \caption{Runtime ratio}
        \label{fig:runtime_ratio}
    \end{subfigure}
    \hfill
    \begin{subfigure}[t]{0.49\linewidth}
        \centering
        \includegraphics[width=\linewidth]{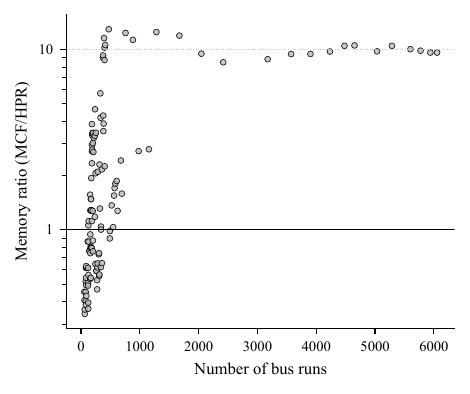}
        \caption{Memory ratio}
        \label{fig:memory_ratio}
    \end{subfigure}

    \caption{Computational performance comparison between MCF and HPR.}
    \label{fig:computational_performance}
\end{figure}

We next compare the pod assignment patterns produced by HPR and MCF. The comparison considers three measures: the number of assignments per pod, the idle time between consecutive assignments, and the empty-movement distance between consecutive assignments.

The number of assignments per pod describes how service work is distributed across the fleet. The Gini coefficient is used to measure how evenly assignments are distributed among pods, with lower values indicating a more even distribution. To compare the two approaches, the Gini gap is defined as
\begin{equation}
\Delta \mathrm{Gini}
=
\mathrm{Gini}_{\mathrm{MCF}}
-
\mathrm{Gini}_{\mathrm{HPR}} .
\end{equation}

Figure~\ref{fig:gini_violin} reports the distribution of $\Delta \mathrm{Gini}$ across the 123 benchmark instances. The median $\Delta \mathrm{Gini}$ is 0.018, with positive values in 80 instances, negative values in 42 instances, and one tie. Thus, HPR produces a more even distribution of assignments in a majority of the paired instances. The Wilcoxon signed-rank test~\citep{wilcoxon1992individual} provides statistical evidence of a positive Gini gap ($p = 0.0026$). However, the median gap of 0.018 indicates that the difference in assignment balance is modest.
\begin{figure}[tb]
    \centering
    \includegraphics[width=0.7\linewidth]{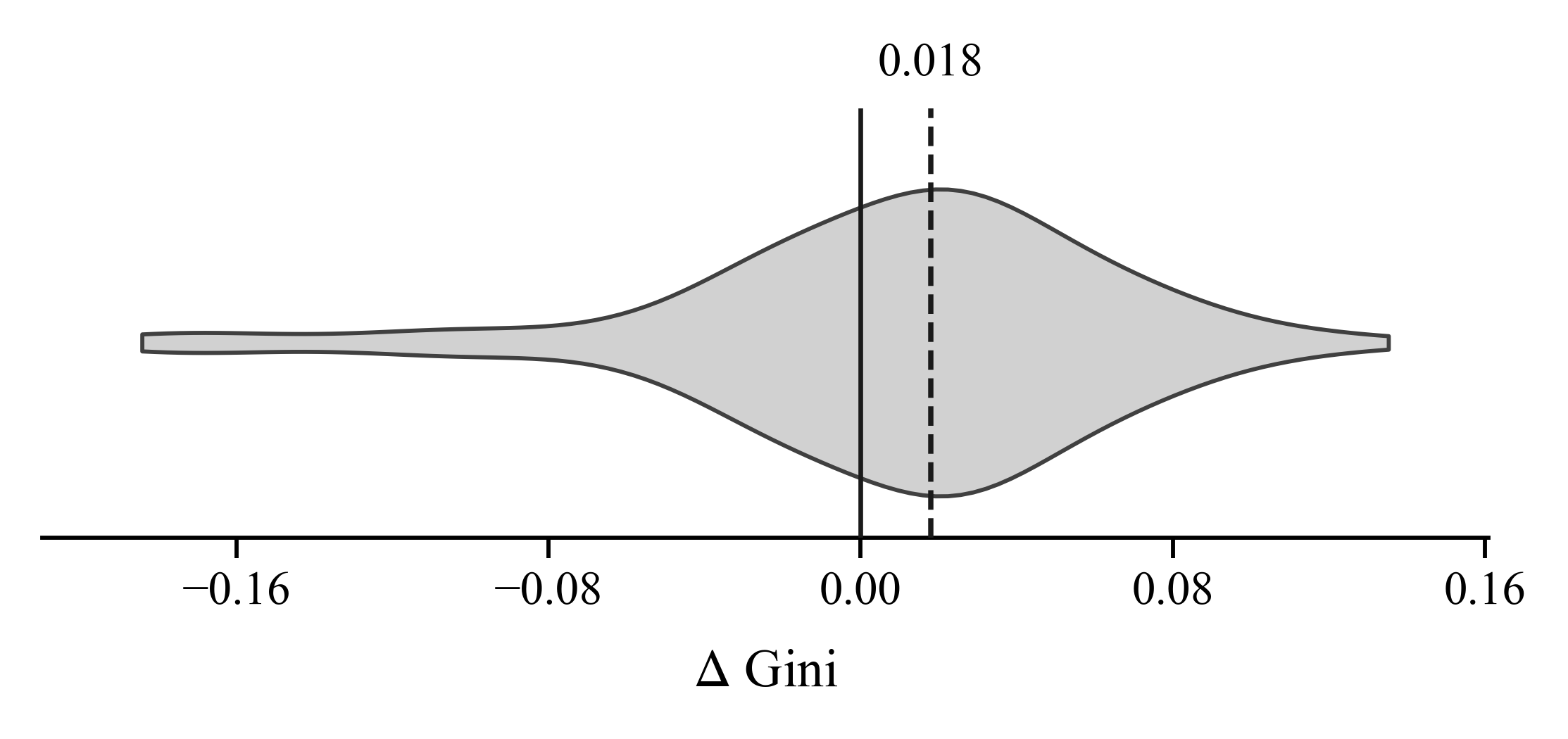}
    \caption{Distribution of the Gini gap for trip counts per pod across instances. The dashed line marks the median.}
   \label{fig:gini_violin}
\end{figure}

Idle time between consecutive assignments measures how long a pod remains unused after completing one pod route and before beginning the next. Table~\ref{tab:idle_time_summary} summarizes the distribution of idle times, in minutes, produced by HPR and MCF. HPR produces shorter typical idle times and a more concentrated distribution, whereas MCF produces a heavier upper tail with more long idle periods.

\begin{table}[tb]
\centering
\caption{Summary of idle times, in minutes, between consecutive pod assignments.}
\small
\label{tab:idle_time_summary}
\begin{tabular}{lrrrrrrrr}
\toprule
Method & Min & P5 & P25 & Median & P75  & P95 & Max \\
\midrule
HPR & 0.00 & 1.98 & 4.98 & 8.98 & 16.62 & 46.78 & 787.00 \\
MCF & 0.00 & 1.00 & 3.25 & 8.00 & 19.00 & 83.00 & 933.00 \\
\bottomrule
\end{tabular}
\end{table}

The distance between consecutive pod assignments measures the deadheading distance from the end of one assignment to the start of the next. Figure~\ref{fig:distance_violin} shows that MCF produces substantially shorter distances between consecutive assignments than HPR. This difference arises because MCF accounts for deadheading costs when linking pod routes, whereas the matching stage of HPR does not.

\begin{figure}[tb]
    \centering
    \includegraphics[width=0.8\linewidth]{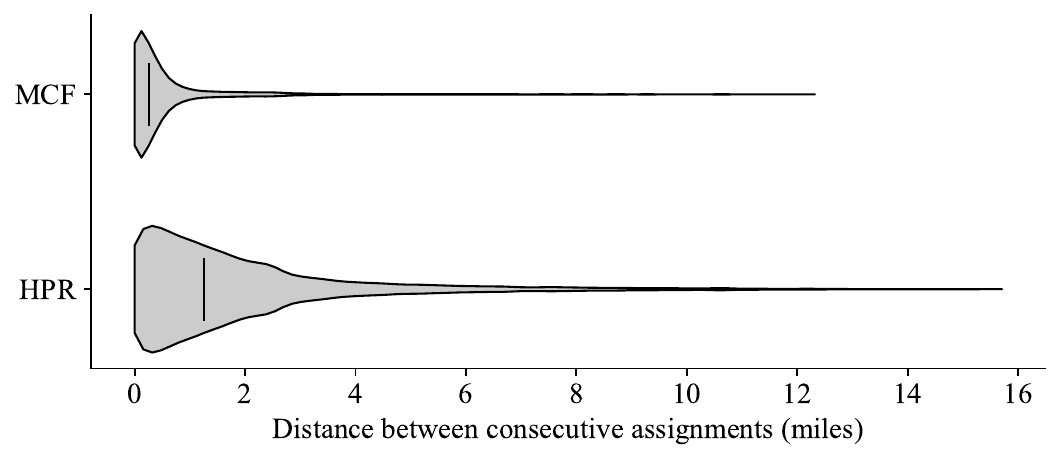}
    \caption{Distribution of distances between consecutive pod assignments.}
    \label{fig:distance_violin}
\end{figure}

\subsection{Sensitivity Analysis}

A sensitivity analysis is conducted by scaling each cost parameter, \(c^{\mathrm{f}}\), \(c^{\mathrm{m}}\), and \(c^{\mathrm{p}}_{s}\), one at a time by a factor \(\beta \in \{0, 0.5, 0.75, 1, 1.25, 1.5, 2\}\), while keeping the other parameters at their baseline values. For each value of \(\beta\), the average optimality gap of HPR relative to MCF is calculated across benchmark instances 109--115.

The left panel of Figure~\ref{fig:sensitivity} shows that the average optimality gap decreases as the fixed pod cost multiplier \(\beta\) increases. Across all tested values of \(\beta\), HPR and MCF deploy the same minimum number of pods. Because both methods deploy the same number of pods, increasing the fixed pod cost adds the same amount to their objective values and leaves the absolute difference between them unchanged. The percentage optimality gap therefore decreases because the common fixed-cost component accounts for a larger share of the MCF objective value.

\begin{figure}[tb]
    \centering
    \includegraphics[width=\linewidth]{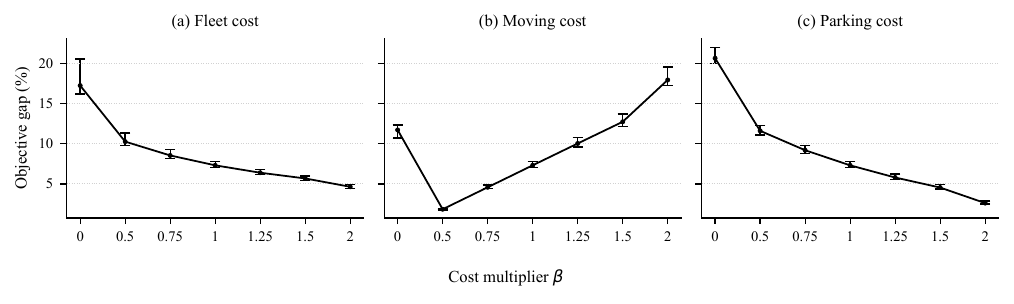}
    \caption{Sensitivity of the HPR optimality gap to the fleet, movement, and parking cost scaling factor $\beta$. Error bars show the range across instances 109--115.}
    \label{fig:sensitivity}
\end{figure}

The middle panel of Figure~\ref{fig:sensitivity} shows that, after an initial decrease from \(\beta=0\) to \(0.5\), the average optimality gap generally increases with the deadheading cost multiplier. HPR and MCF continue to deploy the same number of pods across all tested values of \(\beta\). MCF accounts for deadheading costs when linking pod routes, whereas HPR determines these links through an unweighted matching and considers deadheading costs only after the assignments are fixed. As the deadheading cost increases, the longer movements produced by HPR contribute more heavily to its objective value, increasing the optimality gap.

The right panel of Figure~\ref{fig:sensitivity} shows that the average optimality gap decreases as the parking cost multiplier \(\beta\) increases. HPR and MCF continue to deploy the same number of pods across all tested values of \(\beta\). As shown in Figure~\ref{fig:distance_violin}, HPR produces longer deadheading distances between consecutive assignments than MCF. As the parking cost increases, MCF increasingly substitutes deadheading for parking. Figure~\ref{fig:mcf_parking_moving} illustrates this trade-off in the MCF solutions: parking time decreases monotonically, while deadheading time increases as \(\beta\) increases. The decrease in the optimality gap is therefore associated with the shift in the MCF solutions toward less parking and more deadheading.

\begin{figure}[tb]
    \centering
    \includegraphics[width=0.6\linewidth]{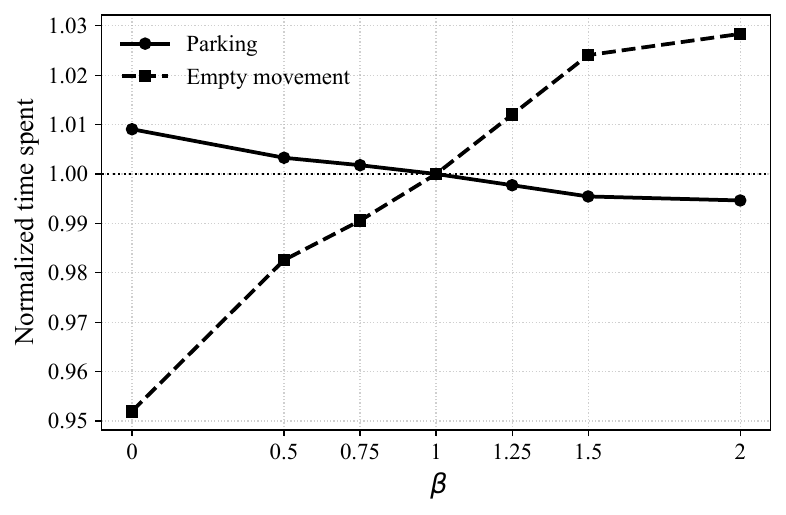}
    \caption{Effect of $\beta$ on parking and empty-movement time in the MCF solution.}
    \label{fig:mcf_parking_moving}
\end{figure}

\section{Conclusion}\label{sec:Conclusion}

This study addresses the operational planning of modular bus systems over a fixed service schedule. The problem requires assigning pods to bus runs and determining how they deadhead and wait between consecutive assignments. A time-space minimum-cost flow formulation is developed to obtain exact solutions and provide a benchmark for evaluating the proposed Hierarchical Pod Routing (HPR) framework. HPR separates the problem into two stages: the first stage constructs feasible pod sequences through a minimum vertex-disjoint path cover, and the second stage selects waiting locations between consecutive assignments based on deadheading and parking costs.

Although developed for modular transit, HPR can also serve as an approximation framework for conventional vehicle scheduling. Future research can extend HPR to account for battery capacity, charging decisions, and energy consumption. The framework can also be extended to incorporate uncertainty in passenger demand and travel times.

\appendix
\section{Proof of Theorem~\ref{Thrm:1}}
\label{app:proof_decomposition}
\begin{proof}
Given the bus run with demands \(\mathbf{d} = (d_{s_1}, \dots, d_{s_n})\), the complete sequential multipartite graph \(G = K_{d_{s_1}, d_{s_2}, \dots, d_{s_n}}\) is constructed. $G$ includes one partition per station, with $d_{s_i}$ nodes for partition $i$. Edges connect every pair of nodes in two consecutive non-empty partitions. A pod route is a path in \(G\) that visits consecutive nodes. A disjoint path cover is a set of such paths that are node-disjoint and together cover all nodes. The objective is to find a disjoint path cover of \(G\) with minimum cardinality, which corresponds to minimizing the number of pods required to serve the demand.

If \(d_{s_i} = 0\) for any \(i\), then \(G\) becomes disconnected into disjoint complete sequential multipartite subgraphs. Let \(H_1, \dots, H_m\) denote the resulting connected components. A minimum disjoint path cover of \(G\) is obtained by taking the union of minimum disjoint path covers of its connected components. Hence, the proof may proceed by focusing on individual connected components. Let \(H\) be a connected component of \(G\) with \(k \geq 1\) non-empty partitions.
\begin{lemma}
There exists a minimum disjoint path cover of \(H\) that contains a path of length \(k - 1\).
\end{lemma}

\begin{proof}
Let \(\mathbf{P} = \{P_1, \dots, P_r\}\) denote an arbitrary minimum disjoint path cover of \(H\). If some path \(P_i \in \mathbf{P}\) already includes one node from each of the \(k\) partitions, the claim is satisfied.

Otherwise, an arbitrary path from \(\mathbf{P}\) is selected and, without loss of generality, denoted by \(P_1\). Suppose \(P_1\) begins at partition \(i\). Since \(H\) is connected and \(\mathbf{P}\) is a disjoint path cover, there exists another path that includes a node from a partition earlier than \(i\). A path that contains a node in a partition earlier than \(i\) is split at partition \(i-1\). Its initial segment is prepended to \(P_1\), and the resulting path remains node-disjoint from the others.

A similar extension is applied at the end of \(P_1\) by identifying a path with a node in a later partition, splitting it, and appending its tail to \(P_1\). By iteratively extending \(P_1\) from both ends, a path is constructed that spans all \(k\) partitions. Each modification keeps the number of paths unchanged and ensures that each node belongs to exactly one path.

Therefore, the resulting collection is also a minimum disjoint path cover of \(H\) and contains a path of length \(k-1\).
\end{proof}
\medskip
Returning to the full graph \(G\), Algorithm~\ref{alg:decompose} constructs such maximum-length paths iteratively. At each step, it selects the earliest unserved station and extends a path through all subsequent stations with available demand. This process corresponds exactly to identifying a longest path in one of the components \(H_i\), removing it from the graph, and repeating the same operation on the residual graph. As shown above, such a longest path always exists in a minimum path cover. Therefore, Algorithm~\ref{alg:decompose} returns a minimum disjoint path cover of \(G\), and thus constructs the minimum possible number of pod routes.
\end{proof}

\bibliographystyle{apalike}
\bibliography{sample}
\end{document}